\documentclass[11pt]{amsart}
\usepackage{amssymb, latexsym}
\theoremstyle{plain}
\newtheorem{theorem}{Theorem}

\newtheorem*{thm-cheb}{Theorem (Chebyshev)}

\newtheorem{proposition}{Proposition}

\newtheorem*{2'}{Theorem 2'}
\newtheorem*{3'}{Theorem 3'}

\theoremstyle{remark}

\newtheorem*{Remark 1}{Remark 1}
\newtheorem*{Remark 2}{Remark 2}
\newtheorem*{Remark 3}{Remark 3}
\newtheorem*{Remark 4}{Remark 4}

\numberwithin{equation}{section}

\begin{document}

\title[A Natural Probabilistic Model on the Integers]
 {A Natural Probabilistic Model on the Integers and its Relation to  Dickman-Type Distributions and Buchstab's Function}

\author{Ross G. Pinsky}


\address{Department of Mathematics\\
Technion---Israel Institute of Technology\\
Haifa, 32000\\ Israel}
\email{ pinsky@math.technion.ac.il}

\urladdr{http://www.math.technion.ac.il/~pinsky/}

\subjclass[2000]{60F05,11N25, 11K65} \keywords{Dickman function, Dickman-type distribution, Buchstab function, prime number, $k$-free numbers }
\date{}

\begin{abstract}
Let $\{p_j\}_{j=1}^\infty$ denote the set of prime numbers in increasing order, let $\Omega_N\subset \mathbb{N}$ denote the set of positive integers with no prime factor larger than $p_N$ and
let $P_N$ denote the probability  measure  on $\Omega_N$ which gives to each $n\in\Omega_N$ a probability proportional to $\frac1n$.
This measure is in fact the distribution of the random integer $I_N\in\Omega_N$ defined by $I_N=\prod_{j=1}^Np_j^{X_{p_j}}$, where
$\{X_{p_j}\}_{j=1}^\infty$ are independent random variables and $X_{p_j}$ is distributed as Geom$(1-\frac1{p_j})$.
We show that $\frac{\log n}{\log N}$ under $P_N$ converges weakly to the Dickman distribution. As a corollary, we recover a classical result from  multiplicative number theory---Mertens'
formula.
Let $D_{\text{nat}}(A)$ denote the natural density of $A\subset\mathbb{N}$, if it exists, and let $D_{\text{log-indep}}(A)=\lim_{N\to\infty}P_N(A\cap\Omega_N)$ denote the
density of $A$ arising from $\{P_N\}_{N=1}^\infty$, if it exists. We show that the two densities coincide on a natural algebra of subsets of $\mathbb{N}$.
We also show that they do not agree on the sets of $n^\frac1s$-\it smooth numbers \rm\ $\{n\in\mathbb{N}: p^+(n)\le n^\frac1s\}$, $s>1$, where $p^+(n)$ denotes the largest prime divisor of $n$.
This last  consideration concerns  distributions involving the Dickman function.
We also consider the
sets of  $n^\frac1s$-\it rough numbers  \rm\ $\{n\in\mathbb{N}:p^-(n)\ge n^{\frac1s}\}$, $s>1$,  where $p^-(n)$ denotes the smallest prime divisor of $n$.
We show that the probabilities of these sets, under
the uniform distribution on $[N]=\{1,\ldots, N\}$ and under the $P_N$-distribution on $\Omega_N$, have the  same
asymptotic decay profile as functions of $s$, although their rates are necessarily different. This profile involves  the Buchstab function. We also prove a new representation for the Buchstab function.

\end{abstract}

\maketitle

\section{Introduction and Statement of  Results}
For a subset $A\subset\mathbb{N}$, the  natural density $D_{\text{nat}}(A)$ of $A$ is defined by
$D_{\text{nat}}(A)=\lim_{N\to\infty}\frac{|A\cap [N]|}N$, whenever this limit exists,
where $[N]=\{1,\ldots,N\}$.
The natural density is  additive, but  not $\sigma$-additive, and therefore  not a measure.
For each prime $p$ and each $n\in\mathbb{N}$, define  the nonnegative integer $\beta_p(n)$, the $p$-adic order of $n$, by $\beta_p(n)=m$, if $p^m\mid n$ and $p^{m+1}\nmid n$. Let  $\delta_p(n)=\max(1,\beta_p(n))$ denote
the indicator function of the set of positive integers divisible by $p$.
It is clear that for each $m\in\mathbb{N}$, the natural density of the set $\{n\in \mathbb{N}: \beta_p(n)\ge m\}$  of natural numbers divisible by $p^m$ is $(\frac1p)^m$.
More generally, it is easy to see that for $l\in\mathbb{N}$,  $\{m_j\}_{j=1}^l\subset\mathbb{N}$ and distinct primes $\{p_j\}_{j=1}^l$,
the natural density of the set $\{n\in\mathbb{N}:\beta_{p_j}(n)\ge m_j, j=1,\ldots, l\}$ is $\prod_{j=1}^l(\frac1{p_j})^{m_j}$.
That is,  the distribution of the random vector $\{\delta_{p_j}\}_{j=1}^l$, defined on the probability space $[N]$ with the uniform distribution,
converges weakly as $N\to\infty$ to the random vector $\{Y_{p_j}\}_{j=1}^l$ with independent components distributed according to the Bernoulli distributions
 $\{\text{Ber}(\frac1{p_j})\}_{j=1}^l$,
and the distribution of the random vector $\{\beta_{p_j}\}_{j=1}^l$
converges weakly as $N\to\infty$ to the random vector $\{X_{p_j}\}_{j=1}^l$ with independent components distributed according to  the geometric distributions Geom$(1-\frac1{p_j})$ \big($P(X_{p_j}=m)=(\frac1{p_j})^m(1-\frac1{p_j})$, $m=0,1,\ldots$\big).
This  fact is the starting point of probabilistic number theory.

Denote the primes in increasing order by $\{p_j\}_{j=1}^\infty$.
In the sequel, we will assume that the random variables $\{X_{p_j}\}_{j=1}^\infty, \{Y_{p_j}\}_{j=1}^\infty$ with distributions as above are defined as independent random variables on some probability space,
and we will use the generic notation $P$ to denote probabilities corresponding to these
random variables.

A real-valued function $f$ defined on $\mathbb{N}$ is called a real arithmetic function.
It is called \it additive\rm\ if $f(nm)=f(n)+f(m)$, whenever $(m,n)=1$. If in addition, $f(p^m)=f(p)$, for all primes $p$ and all $m\ge2$, then it is called \it strongly additive\rm.
Classical examples of  additive arithmetic functions are, for example, $\log \frac{\phi(n)}n$, where $\phi$ is the Euler totient function,
 $\omega(n)$, the number of distinct prime divisors of $n$,   $\Omega(n)$, the number of prime divisors of $n$ counting multiplicities and $\log \sigma(n)$, where  $\sigma$ is the sum-of-divisors function. The first two of these functions are
 strongly additive while the last two are not.

If $f$ is additive, then $f(1)=0$.
Writing $n\in\mathbb{N}$ as $n=\prod_{j=1}^\infty p_j^{\beta_{p_j}(n)}$, we have for $f$ additive,
$f(n)=\sum_{j=1}^\infty f(p_j^{\beta_{p_j}(n)})$,
and for $f$ strongly additive,
$f(n)=\sum_{j=1}^\infty f(p_j^{\delta_{p_j}(n)})=\sum_{j=1}^\infty f(p_j)\delta_{p_j}(n)$.
Equivalently, for each $N\in\mathbb{N}$, we have for $f$ additive,
\begin{equation}\label{likeindepsumadd}
f(n)=\sum_{j=1}^N f(p_j^{\beta_{p_j}(n)}), \ n\in[N],
\end{equation}
and for $f$  strongly additive,
\begin{equation}\label{likeindepsumstrongadd}
f(n)=\sum_{j=1}^N f(p_j)\delta_{p_j}(n),\ n\in[N].
\end{equation}
In light of the above discussion, it is natural to compare
\eqref{likeindepsumadd} to
\begin{equation}\label{indepsumadd}
\mathcal{X}_N\equiv\sum_{j=1}^Nf(p_j^{X_{p_j}}),
\end{equation}
and to compare \eqref{likeindepsumstrongadd} to
\begin{equation}\label{indepsumstrongadd}
\mathcal{Y}_N\equiv\sum_{j=1}^Nf(p_j)Y_{p_j}.
\end{equation}

Now $\mathcal{Y}_N$ converges in distribution as $N\to\infty$ if and only if
it converges almost surely, and the almost sure convergence of $\mathcal{Y}_N$ is characterized
by the Kolmogorov three series theorem \cite{Dur}. Since $EY_{p_j}=EY_{p_j}^2=\frac1{p_j}$, it follows from that theorem that $\mathcal{Y_N}$ converges
almost surely if and only if the following three series converge:
1. $\sum_{j:|f(p_j)|\le1}\frac{f(p_j)}{p_j}$; 2. $\sum_{j:|f(p_j)|\le1}\frac{f^2(p_j)}{p_j}$; 3.
$\sum_{j:|f(p_j)|>1}\frac1{p_j}$.
Since $P(X_{p_j}\ge2)=\frac1{p_j^2}$, it follows from the Borel-Cantelli lemma
that $\sum_{j=1}^\infty 1_{\{X_{p_j}\ge2\}}$ is almost surely finite; thus the very same criterion
also determines whether $\mathcal{X}_N$ converges almost surely.
The Erd\"os-Wintner theorem \cite{EW} states that for additive $f$,
the converges of these three series is  a necessary
and sufficient condition for the convergence in  distribution as $N\to\infty$ of the random variable
$f(n)$ in \eqref{likeindepsumadd}  on the probability space $[N]$ with the uniform distribution.
In the same spirit,  the Kac-Erd\"os theorem \cite{EK} states that if  $f$ is strongly additive and bounded,  then a central limit theorem holds as $N\to\infty$ for
$f(n)$ on the probability space $[N]$ with  the uniform distribution, if the conditions of the Feller-Lindeberg central limit theorem
hold for $\mathcal{Y}_N$. An appropriate corresponding result can be stated for additive $f$ and or unbounded $f$.  There is also a weak law of large numbers result, which in the case
of $f=\omega$ goes by the name of the Hardy-Ramanujan theorem \cite{HR}.
It should be noted that the original proof of Hardy and Ramanujan was quite complicated
and not at all probabilistic; however, the later and much  simpler proof of Turan \cite{Tur} has a strong
probabilistic flavor. For a concise and very readable probabilistic  approach to these results, see Billingsley \cite{Bill}; for a more encyclopedic probabilistic approach, see Elliott \cite{Ell1, Ell2}; for a less probabilistic
approach, see Tenenbaum \cite{Tene}.

Turan's paper with the proof of the Hardy-Ramanujan theorem, as well as
 the Erd\"os-Wintner theorem and several papers leading up to it, all appeared in the 1930's, and
 the Kac-Erd\"os theorem appeared in 1940. Now large deviations for independent and non-identically distributed  random variables have been readily available
 since the 1970's, thus  this author  certainly finds it quite surprising that until very recently no one  extended
 the parallel between \eqref{likeindepsumstrongadd} and \eqref{indepsumstrongadd},
 or \eqref{likeindepsumadd} and \eqref{indepsumadd}, to study the large deviations
 of \eqref{likeindepsumstrongadd} or \eqref{likeindepsumadd}!
See \cite{MZ1, MZ2}.

Another density that is sometimes used in number theory is the \it logarithmic density\rm, $D_{\text{log}}$, which is defined
by
\begin{equation}\label{logden}
D_{\text{log}}(A)=\lim_{N\to\infty}\frac1{\log N}\sum_{n\in A\cap[N]}\frac1n,
\end{equation}
 for  $A\subset \mathbb{N}$, whenever this limit exists.
Using summation by parts, it is easy to show that if $D_{\text{nat}}(A)$ exists, then $D_{\text{log}}(A)$ exists and coincides with
$D_{\text{nat}}(A)$ \cite{Tene}. (On the other hand, there are sets without natural density for which the logarithmic density exists. The most prominent of these are the sets $\{B_d\}_{d=1}^9$  associated with Benford's law, where $B_d$ is the set of positive integers whose first digit is $d$. One has
$D_{\text{log}}(B_d)=\log_{10} (1+\frac1d)$.)
Thus, also on the probability space $[N]$ with the probability measure which gives to each integer
$n$ a    measure proportional to $\frac1n$,
the distribution of the
random vector $\{\beta_{p_j}\}_{j=1}^l$
converges weakly as $N\to\infty$ to the random vector $\{X_{p_j}\}_{j=1}^l$ with independent components distributed according to  the geometric distributions Geom$(1-\frac1{p_j})$.

Motivated by the background described above, in this paper we consider a sequence of probability measures on $\mathbb{N}$ which may be thought
of as a  synthesis between the the logarithmic density $D_{\text{log}}$ and
the concept of approximating the natural density  via a sequence of independent random variables.
Let us denote by
$$
\Omega_N=\{n\in\mathbb{N}: p_j\nmid n, j>N\}
$$
the set of positive integers with no prime divisor larger than $p_N$.
By the Euler product formula,
\begin{equation}\label{CN}
C_N\equiv\sum_{n\in \Omega_N}\frac1n=\prod_{j=1}^N(1-\frac1{p_j})^{-1}<\infty.
\end{equation}
Let $P_N$ denote the probability measure  on $\Omega_N$  for which the probability
of $n$ is proportional to $\frac1n$; namely,
\begin{equation}\label{PN}
P_N(\{n\})=\frac1{C_N}\frac1n,\ n\in\Omega_N.
\end{equation}
The connection between $P_N$ and the logarithmic density is clear; the connection
between $P_N$ and  a sequence of independent random variables comes from the  following proposition.
Define a random positive integer $I_N\in\Omega_N$ by
$$
I_N=\prod_{j=1}^Np_j^{X_{p_j}}.
$$
\begin{proposition}\label{IN}
The distribution of $I_N$ is $P_N$; that is,
$$
P_N(\{n\})=P(I_N=n), \ n\in\Omega_N.
$$
\end{proposition}
\begin{proof}
Let $n=\prod_{j=1}^Np_j^{a_j}\in\Omega_N$. We have
$$
P(I_N=n)=\prod_{j=1}^NP(X_{p_j}=a_j)=\prod_{j=1}^N(\frac1{p_j})^{a_j}(1-\frac1{p_j})=\frac1{C_N}\frac1n=P_N(\{n\}).
$$
\end{proof}

Let $D_{\text{log-indep}}$ denote the asymptotic density obtained from $P_N$:
$$
D_{\text{log-indep}}(A)=\lim_{N\to\infty}P_N(A\cap\Omega_N)=\lim_{N\to\infty}\frac1{C_N}\sum_{n\in A\cap\Omega_N}\frac1n,
$$
for $A\subset\mathbb{N}$, whenever the limit exists.
Note that the weight functions used in calculating
the asymptotic densities $D_{\text{log-indep}}$ and $D_{\text{log}}$ have the same profile, but the sequences
of subsets of $\mathbb{N}$  over which the limits are taken, namely $\{\Omega_N\}_{N=1}^\infty$
and $\{[N]\}_{N=1}^\infty$, are different.
As already noted, when $D_{\text{nat}}(A)$ exists, so does
$D_{\text{log}}(A)$ and they coincide.
We will  show below in Proposition \ref{nat-log-indep}  that the densities $D_{\text{log-indep}}$ and $D_{\text{nat}}$ coincide on  many natural subsets of $\mathbb{N}$.
However we will also show below in Theorem \ref{DickmanforPN} that they disagree on certain important, fundamental subsets of $\mathbb{N}$.

For $k\ge2$, a positive integer $n$ is called \it $k$-free\rm\ if $p^k\nmid n$, for all primes $p$. When $k=2$, one uses the term \it square-free\rm.
Let $S_k$ denote the set of all $k$-free positive integers.
Let
$$
\Omega_N^{(k)}=\Omega_N\cap S_k.
$$
Note that $\Omega_N^{(k)}$ is a finite set; it has $k^N$ elements.
The measure $P_N$ behaves nicely under conditioning on $S_k$. For $k\ge2$, define the measure $P_N^{(k)}$ by
$$
P_N^{(k)}(\cdot)=P_N(~\cdot~|S_k).
$$
Let $\{X^{(k)}_{p_j}\}_{j=1}^\infty$ be independent random variables  with $X^{(k)}_{p_j}$ distributed as $X_{p_j}$ conditioned on $\{X_{p_j}<k\}$.
(Assume that these new random variables are defined on the same space as the $\{X_{p_j}\}_{j=1}^\infty$ so that we can still use $P$ for probabilities.)
Let
$$
I_N^{(k)}=\prod_{j=1}^Np_j^{X_{p_j}^{(k)}}.
$$
\begin{proposition}\label{conddist}
The distribution of $I_N^{(k)}$ is $P_N^{(k)}$.
\end{proposition}
\begin{proof}
$$
P_N^{(k)}(\{n\})=P_N(\{n\}|S_k)=P(I_N=n|X_{p_j}<k,\  j\in[N])=P(I_N^{(k)}=n),
$$
where the second equality follows from Proposition \ref{IN}.
\end{proof}
\bf\noindent Remark.\rm\ The measure $P_N^{(2)}$ was considered by Cellarosi and Sinai in \cite{CS}.
See also the remark after Theorem \ref{Dickman} below.

We will prove the following result, which identifies a certain natural algebra of subsets of $\mathbb{N}$
on which
$D_{\text{log-indep}}$ and $D_{\text{nat}}$ coincide.

\begin{proposition}\label{nat-log-indep}
The densities $D_{\text{log-indep}}$ and $D_{\text{nat}}$ coincide on the algebra of subsets of $\mathbb{N}$ generated by the inverse images of
$\{\beta_{p_j}\}_{j=1}^\infty$ and the sets $\{S_k\}_{k=2}^\infty$.
\end{proposition}

We will show that under the measure $P_N$ as well as under the measure $P_N^{(k)}$,
 the random variable $\frac{\log n}{\log N}$, with $n\in\Omega_N$ in the case
of $P_N$ and $n\in\Omega_N^{(k)}$ in the case of $P_N^{(k)}$, converges in distribution as $N\to\infty$ to the distribution whose
density is $e^{-\gamma}\rho(x)$,    $x\in[0,\infty)$, where
$\gamma$ is Euler's constant, and $\rho$ is the Dickman function, which we now describe.
The Dickman function is  the unique continuous function satisfying
$$
\rho(x)=1,\ x\in(0,1],
$$
 and satisfying the differential-delay equation
$$
x\rho'(x)+\rho(x-1)=0, \ x>1.
$$
By analyzing the Laplace transform of $\rho$, a rather short proof shows
that $\int_0^\infty\rho(x)dx=e^{\gamma}$; thus $e^{-\gamma}\rho(x)$ is indeed a probability
density on $[0,\infty)$. We will call this distribution the \it Dickman distribution\rm.
The distribution decays very rapidly; indeed, it is not hard to show that $\rho(s)\le \frac1{\Gamma(s+1)}$. For an analysis of the Dickman function, see for example, \cite{Tene} or \cite{MV}.
\begin{theorem}\label{Dickman}
Under both $P_N$ and  $P_N^{(k)}$, $k\ge2$, the random variable $\frac{\log n}{\log N}$ converges weakly
to the Dickman distribution.
\end{theorem}
\noindent\bf Remark.\rm\ For $P_N^{(2)}$, Theorem \ref{Dickman} was first proved by
Cellarosi and Sinai \cite{CS}. Their proof involved calculating characteristic functions and was
 quite tedious and long. Our  short proof uses Laplace transforms and the asymptotic
growth rate of the primes given by the Prime Number
Theorem (henceforth PNT).
 After this paper was written, one of the authors of \cite{GM} pointed out to the present author that their paper
also gives a simpler proof of the result in \cite{CS}.
\medskip

Using Theorem \ref{Dickman} we can recover a classical result from multiplicative number theory; namely,

\noindent \bf Mertens' formula.\rm\
\begin{equation}\label{Mertensform}
C_N=\sum_{n\in\Omega_N}\frac1n=\prod_{j=1}^N(1-\frac1{p_j})^{-1}\sim e^{\gamma}\log N,
\ \text{as}\ N\to\infty.
\end{equation}
(Traditionally   the formula is written as $\prod_{p\le N}(1-\frac1p)^{-1}\sim e^{\gamma}\log N$, where the product is over all primes less than or equal to $N$.
 To show that the two are equivalent only requires the fact that $p_N=o(N^{1+\epsilon})$, for any $\epsilon>0$.)
A nice, alternative form of the formula is
$$
\frac{\sum_{n\in\Omega_N}\frac1n}{\sum_{n=1}^N\frac1n}\sim e^\gamma.
$$
Here is the derivation of Mertens' formula from Theorem \ref{Dickman}.
From the definition of $P_N$, we have
$P_N(\frac{\log n}{\log N}\le 1)=\frac1{C_N}\sum_{n=1}^N\frac1n$.
Thus, from Theorem \ref{Dickman}, we have
$\lim_{N\to\infty}\frac1{C_N}\sum_{n=1}^N\frac1n=\int_0^1e^{-\gamma}\rho(x)dx=e^{-\gamma}$.
Now \eqref{Mertensform} follows from this and the fact that $\sum_{n=1}^N\frac1n\sim\log N$.

A direct  proof that  $C_N\sim c\log N$, for some $c$, follows readily with the help of Mertens' second theorem (see \eqref{Mertens2nd}). The proof that the constant is $e^\gamma$ is quite nontrivial.
Of course, our proof of  Mertens' formula  via Theorem \ref{Dickman} uses the fact that   $\int_0^\infty \rho(x)dx=e^\gamma$, but as noted,  this result is obtained readily by analyzing the Laplace transform
of $\rho$.
\medskip

We now present a proof, independent of the  proof we will give later for Theorem \ref{Dickman}, that
 if the limiting distribution of $\frac{\log n}{\log N}$ under $P_N$  exists, then it must be the Dickman distribution. We believe that this is of independent interest.
 Let
$$
J^+_N=\max\{j\in[N]: X_{p_j}\neq0\},
$$
with $\max\emptyset\equiv0$. By Proposition \ref{IN}, the distribution of
$\frac{\log n}{\log N}$ under $P_N$ is  equal to the distribution of
\begin{equation}\label{DNeq}
\begin{aligned}
&D_N\equiv\frac1{\log N}\sum_{n=1}^N X_{p_j}\log p_j=\big(\frac{\log J^+_N}{\log N}\big)\frac1{\log J^+_N}\sum_{j=1}^{J^+_N-1}X_{p_j}\log p_j+\\
&X_{p_{J^+_N}}\frac{\log p_{J^+_N}}{\log N},
\end{aligned}
\end{equation}
where, of course, the sum on the right hand side above is interpreted as equal to 0 if
$J^+_N\le1$, and where we define $p_0=1$.
Our assumption  is that  $\{D_N\}_{N=1}^\infty$ converges weakly to some distribution.
Since $P(J^+_N\le j)=\prod_{m=j+1}^N(1-\frac1{p_m})$,  we have $J^+_N\to\infty$ a.s. as $N\to\infty$.
Also, by the independence of $\{X_{p_j}\}_{j=1}^\infty$, we have $\sum_{j=1}^{J^+_N-1}X_{p_j}\log p_j|\{J^+_N=j_0\}\stackrel{\text{dist}}{=}\sum_{j=1}^{j_0-1}X_{p_j}\log p_j$.
Thus, $\frac1{\log J^+_N}\sum_{j=1}^{J^+_N-1}X_{p_j}\log p_j$
 converges weakly to the same distribution.
Using no more than the weak form of  Merten's formula (namely, $\prod_{j=1}^N(1-\frac1{p_j})^{-1}\sim c\log N$, for some $c$) for the asymptotic equivalence below, we have for $0<x<1$,
\begin{equation}\label{JNasym}
\begin{aligned}
P(\frac{\log J^+_N}{\log N}\le x)=P(J^+_N\le N^x)=\prod_{j=[N^x+1]}^N(1-\frac1{p_j})\sim\frac{\log N^x}{\log N}=x.
\end{aligned}
\end{equation}
Using only the fact that $p_j=o(j^{(1+\epsilon)})$, for any $\epsilon>0$,  it follows that  \eqref{JNasym} also holds with
$\frac{\log J^+_N}{\log N}$ replaced by $\frac{\log p_{J^+_N}}{\log N}$.
Note that $X_{p_{J_N^+}}$ conditioned on $\{J_N^+=j_0\}$ is distributed as $X_{p_{j_0}}$ conditioned on $\{X_{p_{j_0}}\ge1\}$.
A trivial calculation shows  that the conditional distribution of $X_{p_j}$,  conditioned on $X_{p_j}\ge1$, converges weakly to 1 as $j\to\infty$.
From the above facts and \eqref{DNeq} it  follows that if $D$ denotes a random variable distributed according to the limiting distribution of
$\{D_N\}_{N=1}^\infty$, then
\begin{equation}\label{dickmanident}
D\stackrel{\text{dist}}{=}DU+U,
\  \ U\stackrel{dist}{=} \text{Unif}([0,1]),\  U \ \text{and}\ D\ \text{independent}.
\end{equation}
From this, it is a calculus exercise to show that $D$ has a continuous density $f$,  that $f$ is equal to some constant $c$ on $(0,1]$, and that  $f$ satisfies the differential-delay equation satisfied by
the Dickman function $\rho$ on $x>1$. (See, for example, \cite{P}.) Thus $f=c\rho$. Since $f$ is a density and since $\int_0^\infty\rho(x)dx=e^\gamma$, it follows that the density of $D$ is $e^{-\gamma}\rho$.

\medskip

The Dickman function arises in probabilistic number theory in the context of so-called \it smooth \rm\ numbers; that is, numbers all of whose prime divisors are ``small.''
Let $\Psi(x,y)$ denote the number of positive integers less than or equal to $x$ with no
prime divisors greater than $y$. Numbers with no prime divisors greater than $y$ are  called $y$-\it smooth\rm\ numbers.
Then for $s\ge1$,
$\Psi(N,N^\frac1s)\sim N\rho(s)$, as $N\to\infty$.
This result was first proved by Dickman in 1930 \cite{Dick}, whence the name of the function, with later refinements by de Bruijn \cite{deB1}. (In particular, there are rather precise error terms.)
See also \cite{MV} or \cite{Tene}.
Let $p^{+}(n)$ denote the largest prime divisor of $n$. Then Dickman's result states that
 the random variable $\frac{\log N}{\log p^{+}(n)}$ on the probability space  $[N]$ with the uniform distribution converges weakly in distribution as $N\to\infty$ to
the distribution whose distribution function is $1-\rho(s)$, $s\ge1$, and whose density is $-\rho'(s)=\frac{\rho(s-1)}s,\  s\ge1$.
Since $\frac{\log n}{\log N}$ on the probability space  $[N]$ with the uniform distribution converges weakly in distribution to 1 as $N\to\infty$, an equivalent statement of Dickman's
result is that the random variable $\frac{\log n}{\log p^{+}(n)}$ on the probability space  $[N]$ with the uniform distribution converges weakly in distribution as $N\to\infty$ to
the distribution whose distribution function is $1-\rho(s)$, $s\ge1$,
For later use,
we state this as follows in terms of the natural density:
\begin{equation}\label{Dickthm}
D_{\text{nat}}(\{n\in\mathbb{N}:p^{+}(n)\le n^\frac1s\})=D_{\text{nat}}(\{n\in\mathbb{N}:\frac{\log n}{\log p^{+}(n)}\ge s\})=\rho(s),\ s\ge1.
\end{equation}
We will call $\{n\in\mathbb{N}:p^{+}(n)\le n^\frac1s\}$ the set of \it\ $n^\frac1s$-smooth numbers.\rm\

The standard number-theoretic proof of  Dickman's result is via induction. It can be checked that this inductive proof also works to obtain a corresponding result for $k$-free integers. Thus,
\begin{equation}\label{DickthmSk}
\begin{aligned}
&D_{\text{nat}}(\{n\in\mathbb{N}:p^{+}(n)\le    n^\frac1s\}|S_k)=D_{\text{nat}}(\{n\in\mathbb{N}:\frac{\log n}{\log p^{+}(n)}\ge s\}|S_k)=\rho(s),\\
&\text{for}\ s\ge1 \ \text{and}\ k\ge2.
\end{aligned}
\end{equation}

\noindent \bf Remark.\rm\ Equivalent to \eqref{Dickthm} is the statement that   $\frac{\log p^{+}(n)}{\log n}$ on $[N]$ with the uniform distribution converges weakly in distribution as $N\to\infty$ to
the distribution whose distribution function is $\rho(\frac1s),\ s\in[0,1]$. The corresponding density function is then $\frac{-\rho'(\frac1s)}{s^2}=\frac1s\rho(\frac1s-1)$.
In the spirit of \eqref{dickmanident}, it has been shown that if $\hat D$ denotes a random variable with the above distribution, then
$$
\hat D\stackrel{dist}{=}\max(1-U,\hat Du),\ \ U\stackrel{dist}{=}\text{Unif}([0,1]),\  U \ \text{and}\ \hat D\ \text{independent}.
$$
In light of the comparison between \eqref{dickmanident} and the above equation, the distribution has been dubbed the \it  max-Dickman distribution\rm\ \cite{PW}.
This distribution is the first coordinate of the Poisson-Dirichlet distribution on the infinite simplex
$\{x=(x_1,x_2,\ldots): x_i\ge0, \sum_{i=1}^\infty x_i=1\}$. The Poisson-Dirichlet distribution can be defined as the decreasing order statistics of the GEM distribution, where the GEM
distribution is the ``stick-breaking'' distribution: let $\{U_n\}_{n=1}^\infty$ be IID uniform variables on $[0,1]$; let $Y_1=U_1$, and let $Y_n=U_n\prod_{r=1}^{n-1}(1-U_r)$, $n\ge2$;
then $(Y_1,Y_2,\ldots)$ has the GEM distribution.
The $n$-dimensional density function for the distribution of the first $n$ coordinates of the Poisson-Dirichlet distribution is given by
\begin{equation*}\label{PD-ncoord}
\begin{aligned}
&f^{(n)}(s_1,s_2,\ldots, s_n)=\frac1{s_1\cdots s_n}\rho(\frac{1-s_1-\cdots-s_n}{s_n}),\\
&\text{for}\ 0<s_n<\cdots<s_1<1\ \text{and}\ \sum_{j=1}^ns_j<1.
\end{aligned}
\end{equation*}
Let $p^+_j(n)$ denote the $j$th largest distinct prime divisor of $n$, with $p^+_j(n)=1$ if $n$ has fewer than $j$ distinct prime divisors.
In 1972 Billingsley \cite{Bill72} gave a probabilistic proof of the fact  that $\frac1{\log n}(\log p^+_1(n),\log p^+_2(n),\ldots)$ on $[N]$ with the uniform distribution converges weakly in distribution as $N\to\infty$ to
the Poission-Dirichlet distribution. However,  he did not identify it as such as the the theory of the Poisson-Dirichlet distribution had not yet been developed.
(We note that the random vector consisting of the lengths of the cycles of a uniformly random permutation of $[N]$, arranged in decreasing order, when normalized by dividing their lengths by $N$, also converges as $N\to\infty$
to the Poisson-Dirichlet distribution \cite{ABT}.)
\medskip

Proposition \ref{nat-log-indep} shows that $D_{\text{nat}}$ and $D_{\text{log-indep}}$
coincide on a certain natural algebra of sets. We will prove  that they disagree on the sets appearing in \eqref{Dickthm} or \eqref{DickthmSk}; namely on
the sets of $n^\frac1s$-smooth numbers, $s>1$, and on the intersection of such a set with the set of $k$-free numbers, $S_k$, $k\ge2$.

\begin{theorem}\label{DickmanforPN}
Under both $P_N$ and $P_N^{(k)}$  the random variable $\frac{\log n}{\log p^{+}(n)}$  converges weakly  as $N\to\infty$ to
$D+1$, where $D$ has the  Dickman distribution; that is,
\begin{equation}\label{PNDick}
\begin{aligned}
&D_{\text{log-indep}}(\{n\in\mathbb{N}:p^{+}(n)\le n^\frac1s\})=D_{\text{log-indep}}(\{n\in\mathbb{N}:\frac{\log n}{\log p^{+}(n)}\ge s\})=\\
&e^{-\gamma}\int_{s-1}^\infty\rho(x)dx,\ s\ge1;
\end{aligned}
\end{equation}
\begin{equation*}\label{PNDickSk}
\begin{aligned}
&D_{\text{log-indep}}(\{n\in\mathbb{N}:p^{+}(n)\le n^\frac1s\}|S_k)=D_{\text{log-indep}}(\{n\in\mathbb{N}:\frac{\log n}{\log p^{+}(n)}\ge s\}|S_k)=\\
&e^{-\gamma}\int_{s-1}^\infty\rho(x)dx,\ s\ge1, k\ge2.
\end{aligned}
\end{equation*}
\end{theorem}
\noindent\bf Remark.\rm\ Recalling that whenever the natural density exists, the logarithmic one does too and they are equal, it follows  from \eqref{Dickthm}
that $D_{\text{log}}(\{n\in\mathbb{N}:p^{+}(n)\le n^\frac1s\})=\rho(s)$.
Since, as we've noted, the weights used in calculating the densities  $D_{\text{log}}$ and $D_{\text{log-indep}}$ have the same profile, but the sequences
of subsets of $\mathbb{N}$  over which the limits are taken, namely $\{[N]\}_{N=1}^\infty$ and $\{\Omega_N\}_{N=1}^\infty$, are different,
and since the integers in $[N]$ and in $\Omega_N$ are constructed from the same set $\{p_j\}_{j=1}^N$ of primes, and $[N]\subset\Omega_N$,
intuition suggests that
\begin{equation}\label{comparedickmans}
\rho(s)\le e^{-\gamma}\int_{s-1}^\infty\rho(x)dx, \  s\ge1;
\end{equation}
 that is, that under $D_{\text{log-indep}}$, $n^\frac1s$-smooth numbers are more likely than under
$D_{\text{nat}}$. And indeed this is the case. Letting $H(s)=e^{-\gamma}\int_{s-1}^\infty\rho(x)dx-\rho(s)$, we have $H(1)=H(\infty)=0$. Differentiating
$H$, and using the differential-delay equation satisfied by $\rho$, one has $H'(s)=-e^{-\gamma}\rho(s-1)-p'(s)=\rho(s-1)(\frac1s-e^{-\gamma})$.
Thus, $H'(s)$ vanishes only at $s=e^{\gamma}$. Differentiating again and again using the differential-delay equation, one finds that $H''(e^\gamma)<0$; thus,
$H(s)\ge0$, for $s\ge1$, proving \eqref{comparedickmans}.
\medskip

We now consider  integers all of whose prime divisors are ``large.'' Let $\Phi(x,y)$ denote the number of positive integers less than or equal to $x$ all of whose prime divisors
are greater than or equal to $y$. Numbers with no prime divisors less than $y$ are
called $y$-\it rough\rm\ numbers.
The Buchstab function $\omega(s)$, defined for $s\ge1$, is  the unique continuous function satisfying
$$
\omega(s)=\frac1s, \ 1\le s\le 2,
$$
 and  satisfying the differential-delay equation
$$
(s\omega(s))'=\omega(s-1),\ s>2.
$$
In 1937, Buchstab proved \cite{Buch} that for $s>1$,
$\Phi(N,N^\frac1s)\sim\frac{Ns\omega(s)}{\log N}$ as $N\to\infty$; whence the name of the function.
See also \cite{MV} or \cite{Tene}.
Let $p^{-}(n)$ denote the smallest prime divisor of $n$.
Then Buchstab's result states that
\begin{equation}\label{Buch}
\begin{aligned}
&\frac{|\{n\in[N]: p^{-}(n)\ge N^\frac1s\}|}N=\frac{|\{n\in[N]: \frac{\log N}{\log p^-(n)}\le s\}|}N
\sim\frac{s\omega(s)}{\log N},\\
&\text{for}\ s>1,
\ \text{as}\ N\to\infty.
\end{aligned}
\end{equation}
Since $\frac{|\{n\in[N]:\frac{\log N}{\log n}>1+\epsilon\}|}N=N^{-\epsilon}$, it follows that \eqref{Buch} is equivalent to
 \begin{equation}\label{Buchagain}
\begin{aligned}
&\frac{|\{n\in[N]: p^{-}(n)\ge n^\frac1s\}|}N=\frac{|\{n\in[N]: \frac{\log n}{\log p^-(n)}\le s\}|}N
\sim\frac{s\omega(s)}{\log N},\\
&\text{for}\ s>1,
\ \text{as}\ N\to\infty.
\end{aligned}
\end{equation}
One has $\lim_{s\to\infty}\omega(s)=e^{-\gamma}$, and  the rate of convergence is super-exponential \cite{Tene}.
We will call $\{n\in[N]: p^{-}(n)\ge n^\frac1s\}$ the set of \it $n^\frac1s$-rough numbers.\rm\
(We note that the probability that the shortest  cycle of a uniformly random permutation of $[N]$ is larger or equal to $\frac Ns$  decays asymptotically as $\frac{s\omega(s)} N$.)

Note that  \eqref{Buchagain} also holds for $s=1$, since in this case \eqref{Buchagain} reduces to $\frac{\Pi(N)}N\sim\frac 1{\log N}$; that is, it reduces
to the  PNT. Buchstab assumed the PNT in proving \eqref{Buch}.

What is the asymptotic  probability of a prime number under the sequence of measures used to construct the logarithmic density $D_{\text{log}}$ and under the sequence $\{P_N\}_{N=1}^\infty$ used to construct
the density $D_{\text{log-indep}}$?
Mertens' second theorem states that
\begin{equation}\label{Mertens2nd}
\sum_{p\le N}\frac1p=\log\log N+M_0+O(\frac1{\log N}),
\end{equation}
where the summation is over primes $p$, and where $M_0$ is called  the Meissel-Mertens constant \cite{N}.
By the PNT, $p_N\sim N\log N$, thus by Mertens' second theorem,
\begin{equation}\label{Mertens2}
\sum_{j=1}^N\frac1{p_j}\sim\log\log (N\log N)\sim\log\log N.
\end{equation}
From \eqref{Mertens2nd} we conclude that for the sequence of measures used to construct the logarithmic density $D_{\text{log}}$, the probability of a prime is
\begin{equation}\label{logdensity}
\frac1{\log N}\sum_{p\le N}\frac1p\sim \frac{\log\log N}{\log N}.
\end{equation}
Since
\begin{equation*}\label{PNPNprelim}
P_N(\{n\in\Omega_N: n\ \text{is prime}\})=\frac1{C_N}\sum_{j=1}^N\frac1{p_j},
\end{equation*}
 from   \eqref{Mertens2} and  Mertens formula given  in \eqref{Mertensform}, we conclude that
for the sequence  $\{P_N\}_{N=1}^\infty$ use to construct the density $D_{\text{log-indep}}$, the probability of a prime satisfies
\begin{equation}\label{PNPN}
P_N(\{n\in\Omega_N: n\ \text{is prime}\})\sim \frac{e^{-\gamma}\log \log N}{\log N}.
\end{equation}

From \eqref{logdensity} and \eqref{PNPN} it is clear that \eqref{Buchagain} cannot hold when the sequence of uniform measures on $[N]$, $N=1,2,\ldots$, appearing on the left hand side there is replaced either by the sequence of
measures
used to calculate the logarithmic density  $D_{\text{log}}$ or by the sequence $\{P_N\}_{N=1}^\infty$ used to calculate the density $D_{\text{log-indep}}$.
However, letting
$$
a_s(n)=\begin{cases} 1,\  p^-(n)\ge n^{\frac1s},\\ 0,\  \text{otherwise},\end{cases},
$$
and $A_s(t)=\sum_{j=1}^{[t]}a_s(j)$, $t\ge1$,
a summation by parts gives
\begin{equation}\label{sumbyparts}
\sum_{n\le N: p^-(n)\ge n^{\frac1s}}\frac1n=\sum_{n=1}^N\frac{a_s(n)}n=\frac{A_s(N)}N+\int_1^N\frac{A_s(t)}{t^2}dt.
\end{equation}
By \eqref{Buchagain}, $\frac{A_s(t)}t\sim\frac{s\omega(s)}{\log t}$ as $t\to\infty$; thus from
\eqref{sumbyparts} we have
$$
\frac1{\log N}\sum_{n\le N: p^-(n)\ge n^{\frac1s}}\frac1n\sim\log\log N~\frac{s\omega(s)}{\log N}.
$$
That is, modulo the change necessitated by  comparing \eqref{logdensity} to the PNT, Buchstab's result on $n^\frac1s$-rough numbers for the uniform measure in \eqref{Buchagain} carries over to the measures used
in the construction of the logarithmic density.

Modulo the change necessitated by comparing  \eqref{PNPN} to the PNT, does Buchstab's result on $n^\frac1s$-rough numbers
also carry over to the measures $\{P_N\}_{N=1}^\infty$ used in the construction of the density $D_{\text{log-indep}}$?
Since \eqref{Dickthm} and  \eqref{PNDick} show that the  \it positive\rm\  densities with respect $D_{\text{nat}}$ and $D_{\text{log-indep}}$ of the $n^\frac1s$-smooth sets $\{n\in\mathbb{N}:p^{+}(n)\le n^\frac1s\}$
do not coincide,
 it is interesting to discover that the  answer is indeed affirmative.
\begin{theorem}\label{BuchstabforPN}
For $s\ge1$,
\begin{equation}\label{Buchstab}
\begin{aligned}
&P_N(\{n\in[N]: p^{-}(n)\ge n^\frac1s\})=P_N(\frac{\log n}{\log p^-(n)}\le s)
\sim (e^{-\gamma}\log\log N)~\frac{s\omega(s)}{\log N},\\
& \text{as}\ N\to\infty.
\end{aligned}
\end{equation}
\end{theorem}
Recalling the definition of the Buchstab function,
note that $V(s)\equiv s\omega(s)$ is the unique continuous function satisfying
$V(s)=1$, $1\le s\le 2$, and $V'(s)=\frac{V(s-1)}{s-1}$, for $s>2$.
In the proof of Theorem \ref{BuchstabforPN}, we actually show that \eqref{Buchstab} holds with
$s\omega(s)$ on the right hand side replaced by
\begin{equation*}
v(s)\equiv\sum_{L=1}^{[s]}\Lambda_L(s),
\end{equation*}
where
\begin{equation}\label{Lambda}
\begin{aligned}
&\Lambda_1(s)=1,\ s\ge1;\\
 &\Lambda_2(s)=\int_1^{s-1}\frac{du_1}{u_1}=\log(s-1), \ s\ge2;\\
&\Lambda_L(s)=\int_{L-1}^{s-1}\int_{L-2}^{u_{L-1}-1}\cdots\int_1^{u_2-1}\prod_{j=1}^{L-1}
\frac{du_j}{u_j},\ s\ge L\ge3.
\end{aligned}
\end{equation}
Now  $\Lambda_L'(s)=\frac1{s-1}\Lambda_{L-1}(s-1)$, for $s\ge L\ge2$, while of course
$\Lambda_1'(s)=0$.
Thus, $v(s)=1$, for $1\le s\le 2$ and $v'(s)=\frac{v(s-1)}{s-1}$, for $s>2$.
This proves the following result.
\begin{proposition}\label{Buchrep}
\begin{equation}\label{newBuch}
s\omega(s)=1+\log(s-1)+\sum_{L=3}^{[s]}\int_{L-1}^{s-1}\int_{L-2}^{u_{L-1}-1}\cdots\int_1^{u_2-1}\prod_{j=1}^{L-1}
\frac{du_j}{u_j}, \ s\ge3.
\end{equation}
\end{proposition}
The representation of the Buchstab function $\omega$ in \eqref{newBuch}  seems to be  new.
It is simpler than the following known representation \cite{ABT, L}:
$$
s\omega(s)=1+\sum_{L=2}^{[s]}\frac1{L!}\int_{\stackrel{\frac1s\le y_j\le 1}{\frac1s\le 1-(y_1+y_2+\cdots+y_{L-1})\le 1)}}\frac1{1-(y_1+y_2+\cdots+y_{L-1})}\prod_{j=1}^{L-1}
\frac{dy_j}{y_j}.
$$
Since  $\lim_{s\to\infty}\omega(s)=e^{-\gamma}$, we also obtain what seems to be
yet another representation of Euler's constant:
$$
e^{-\gamma}=\lim_{N\to\infty}\frac1N\sum_{L=3}^{N}\int_{L-1}^N\int_{L-2}^{u_{L-1}-1}\cdots\int_1^{u_2-1}\prod_{j=1}^{L-1}
\frac{du_j}{u_j}.
$$

We prove Proposition \ref{nat-log-indep} and Theorems \ref{Dickman}-\ref{BuchstabforPN} successively in sections
2-5 below.

\section{Proof of Proposition \ref{nat-log-indep}}
For the proof of the proposition we need the following result which is obviously known; however, as
we were unable to find it in a number theory text, we  supply a proof in the appendix.

\begin{proposition}\label{Skbetap}
For $1\le l<k$,
\begin{equation}\label{Skbetapj}
D_{\text{nat}}(\beta_{p_j}\ge l|S_k)\equiv\frac{D_{\text{nat}}(\{\beta_{p_j}\ge l\}~\cap S_k)}{D_{\text{nat}}(S_k)}=\frac{(\frac1{p_j})^l-(\frac1{p_j})^k}{1-(\frac1{p_j})^k}.
\end{equation}
\end{proposition}
\bf\noindent  Remark.\rm\
When $k=2$ and $l=1$, \eqref{Skbetapj} becomes \newline
$D_{\text{nat}}(\beta_{p_j}\ge1|S_2)=\frac1{1+p_j}$. That is, among square-free numbers, the natural density of those divisible by the prime $p_j$ is $\frac1{p_j+1}$.
\medskip

\noindent \it Proof of Proposition \ref{nat-log-indep}.\rm\
In light of Proposition \ref{IN}, it follows immediately that for $l\le N$, the random vector $\{\beta_{p_j}\}_{j=1}^l$ under $P_N$
has the distribution of $\{X_{p_j}\}_{j=1}^l$ under $P$, this latter distribution being the weak limit as $N\to\infty$ of the  distribution of
$\{\beta_{p_j}\}_{j=1}^l$ on $[N]$ with the uniform distribution.
From this it follows that $D_{\text{log-indep}}$ and $D_{\text{nat}}$ coincide on the algebra of  sets  generated by the inverse images
of the $\{\beta_{p_j}\}_{j=1}^\infty$.

It is well-known that $D_{\text{nat}}(S_k)=\frac1{\zeta(k)}$, where $\zeta(s)=\sum_{n=1}^\infty\frac1{n^s}$ is the Riemann zeta function \cite{S}.
On the other hand, by Proposition \ref{IN} we have
$$
P_N(S_k)=P(X_{p_j}<k, j\in[N])=\prod_{j=1}^N P(X_j<k)=\prod_{j=1}^N(1-\frac1{p_j^k}),
$$
and so by the Euler product formula we conclude that
$$
D_{\text{log-indep}}(S_k)=\lim_{N\to\infty}P_N(S_k)=\lim_{N\to\infty}\prod_{j=1}^N(1-\frac1{p_j^k})=
\frac1{\zeta(k)}.
$$
Thus, the two densities coincide on the algebra generated by $\{S_k\}_{k=2}^\infty$.

Also, for $j\le N$, $k\ge2$  and  $l<k$, we have
\begin{equation*}\label{SkAp}
\begin{aligned}
&P_N^{(k)}(\beta_{p_j}\ge l)=P_N(\beta_{p_j}\ge l|S_k)=P\big(X_{p_j}\ge l|X_{p_i}<k, i=1,\ldots, N\big)=\\
&P(X_{p_j}\ge l|X_{p_j}<k)=
\frac{\sum_{i=l}^{k-1}(\frac1{p_j})^i(1-\frac1{p_j})}
{\sum_{i=0}^{k-1}(\frac1{p_j})^i(1-\frac1{p_j})}=\frac{(\frac1{p_j})^l-(\frac1{p_j})^k}{1-(\frac1{p_j})^k}.
\end{aligned}
\end{equation*}
Thus,
$D_{\text{log-indep}}(\beta_{p_j}\ge l|S_k)\equiv\frac{D_{\text{log-indep}}(\{\beta_{p_j}\ge l\}\cap S_k)}{D_{\text{log-indep}}(S_k)}=\frac{(\frac1{p_j})^l-(\frac1{p_j})^k}{1-(\frac1{p_j})^k}$.
Recalling Proposition \ref{Skbetap},
we conclude that the two densities indeed coincide on the algebra generated by the inverse images of $\{\beta_{p_j}\}_{j=1}^\infty$ and the sets $\{S_k\}_{k=2}^\infty$.
\hfill$\square$

\section{Proof of Theorem \ref{Dickman}}\label{sec-Dick}
We first prove the theorem for $P_N$. Let $E_N$ denote the expectation with respect to $P_N$.
Using Proposition \ref{IN}, we have
$$
E_N\frac{\log n}{\log N}=\frac1{\log N}\sum_{j=1}^NEX_{p_j}\log p_j=\frac1{\log N}\sum_{j=1}^N\frac{\log p_j}{p_j-1}.
$$
Mertens' first theorem \cite{N} states that $\sum_{p\le N}\frac{\log p}p\sim\log N$, where the sum is over all primes less than or equal to $N$.
Thus, using nothing more than the trivial bound $p_N\le N^k$, for some $k$, it follows that
$\{E_N\frac{\log n}{\log N}\}_{N=1}^\infty$ is bounded, and therefore
  that the distributions of the nonnegative random variables  $\{\frac{\log n}{\log N}\}_{N=1}^\infty$ under $\{P_N\}_{N=1}^\infty$
 are tight. In the next paragraph we will prove that their Laplace transforms converge
 to $\exp(-\int_0^1\frac{1-e^{-tx}}xdx)$.
 This proves that the distributions converge weakly. By the argument in  the paragraph containing \eqref{DNeq}, it then follows that the limiting distribution is the Dickman distribution.
 Alternatively,  the above function is known to be the Laplace transform of the Dickman distribution \cite{MV, Tene}.

By Proposition \ref{IN},  we have for $t\ge0$,
\begin{equation}\label{LT}
E_N\exp(-t\frac{\log n}{\log N})=E\exp(-\frac t{\log N}\sum_{j=1}^NX_{p_j}\log p_j)=
\prod_{j=1}^NE\exp(-\frac {t\log p_j}{\log N}X_{p_j}).
\end{equation}
For $s\ge0$,
\begin{equation}\label{calc1}
\begin{aligned}
&E\exp(-sX_{p_j})=\sum_{k=0}^\infty e^{-sk}(\frac1{p_j})^k(1-\frac1{p_j})=
(1-\frac1{p_j})\frac1{1-\frac{e^{-s}}{p_j}}=
\frac1{1+\frac{1-e^{-s}}{p_j-1}}.
\end{aligned}
\end{equation}
From \eqref{LT} and \eqref{calc1} we have
\begin{equation}
\log E_N\exp(-t\frac{\log n}{\log N})=-\sum_{j=1}^N\log\big(1+\frac{1-\exp(-t\frac{\log p_j}{\log N})}{p_j-1}\big).
\end{equation}
Now $x-\frac{x^2}2\le\log (1+x)\le x$, for $x\ge0$, and by the bounded convergence theorem,
$\lim_{N\to\infty}\sum_{j=1}^N\Big( \frac{1-\exp(-t\frac{\log p_j}{\log N})}{p_j-1}\Big)^2=0$;
thus,
\begin{equation}\label{loglim}
\lim_{N\to\infty}\log E_N\exp(-t\frac{\log n}{\log N})=-\lim_{N\to\infty}\sum_{j=1}^N
 \frac{1-\exp(-t\frac{\log p_j}{\log N})}{p_j-1}.
\end{equation}

Let $x^{(N)}_j=\frac{\log p_j}{\log N}$ and  $\Delta^{(N)}_j=x^{(N)}_{j+1}-x^{(N)}_j$.
By the PNT, $p_j\sim j\log j$, as $j\to\infty$; thus
\begin{equation}\label{keyforchebthm}
\log p_{j+1}-\log p_j\sim\log\frac{(j+1)\log(j+1)}{j\log j}=\log\big((1+\frac1j)(1+\frac{\log(1+\frac1j)}{\log j})\big)\sim\frac1j\sim\frac{\log p_j}{p_j}.
\end{equation}
Consequently,
\begin{equation}\label{mesh}
\Delta^{(N)}_j\sim\frac{
\log p_j}{p_j\log N},\ \text{uniformly  as }\  j,N\to\infty.
\end{equation}
 Note also that
\begin{equation}\label{endpts}
\lim_{N\to\infty} x^{(N)}_1=0, \ \lim_{N\to\infty}x^{(N)}_N=1.
\end{equation}
We rewrite  the summand on the right hand side of \eqref{loglim}
as
\begin{equation}\label{R-int}
\begin{aligned}
&\sum_{j=1}^N
 \frac{1-\exp(-t\frac{\log p_j}{\log N})}{p_j-1}=\sum_{j=1}^N
 \frac{1-\exp(-t\frac{\log p_j}{\log N})}{\frac{\log p_j}{\log N}} \frac{\log p_j}{(p_j-1)\log N}=\\
& \sum_{j=1}^N
 \frac{1-\exp(-tx^{(N)}_j)}{x^{(N)}_j} \frac{\log p_j}{(p_j-1)\log N}.
\end{aligned}
\end{equation}
From \eqref{mesh}-\eqref{R-int} along with \eqref{loglim} we conclude that
\begin{equation}\label{finalthm1}
\lim_{N\to\infty}E_N\exp(-t\frac{\log n}{\log N})=\exp(-\int_0^1\frac{1-e^{-tx}}xdx).
\end{equation}
This completes the proof of the  theorem for $P_N$.

We now turn to $P_N^{(k)}$.  Let $E_N^{(k)}$ denote the expectation with respect to $P_N^{(k)}$.
By Proposition \ref{conddist},
\begin{equation}\label{LT-k}
E^{(k)}_N\exp(-t\frac{\log n}{\log N})=E\exp(-\frac t{\log N}\sum_{j=1}^NX^{(k)}_{p_j}\log p_j)=
\prod_{j=1}^NE\exp(-\frac {t\log p_j}{\log N}X^{(k)}_{p_j}).
\end{equation}
For $s\ge0$,
\begin{equation}\label{calc2}
\begin{aligned}
&E\exp(-sX^{(k)}_{p_j})=\sum_{l=0}^{k-1} e^{-sl}(\frac1{p_j})^l\frac{1-\frac1{p_j}}{1-(\frac1{p_j})^k}=
\frac{1-\frac1{p_j}}{1-(\frac1{p_j})^k}
\frac{1-(\frac{e^{-s}}{p_j})^k}{1-\frac{e^{-s}}{p_j}}.
\end{aligned}
\end{equation}
Comparing the equality between the first and third expressions in
\eqref{calc1} with \eqref{calc2}, we have
\begin{equation}\label{EsEsk}
E\exp(-sX^{(k)}_{p_j})=\frac{1-(\frac{e^{-s}}{p_j})^k}{1-(\frac1{p_j})^k}E\exp(-sX_{p_j})=\Big(1+\frac{(\frac1{p_j})^k(1-e^{-sk})}{1-(\frac1{p_j})^k}\Big)E\exp(-sX_{p_j}).
\end{equation}
Thus, from \eqref{LT}, \eqref{LT-k} and \eqref{EsEsk} we have
\begin{equation}\label{compareEEk}
E^{(k)}_N\exp(-t\frac{\log n}{\log N})=E_N\exp(-t\frac{\log n}{\log N})\prod_{j=1}^N\big(1+\frac{(\frac1{p_j})^k\big(1-\exp(-\frac{kt\log p_j}{\log N})\big)}{1-(\frac1{p_j})^k}\big).
\end{equation}
By the bounded convergence theorem,
\begin{equation}\label{negl}
\lim_{N\to\infty}\sum_{j=1}^N\frac{(\frac1{p_j})^k\big(1-\exp(-\frac{kt\log p_j}{\log N})\big)}{1-(\frac1{p_j})^k}=0.
\end{equation}
Thus, from \eqref{finalthm1}, \eqref{compareEEk} and \eqref{negl}, we conclude that
\begin{equation*}\label{finalthm1k}
\lim_{N\to\infty}E_N^{(k)}\exp(-t\frac{\log n}{\log N})=\exp(-\int_0^1\frac{1-e^{-tx}}xdx).
\end{equation*}
\hfill $\square$
\section{Proof of Theorem \ref{DickmanforPN}}
We prove the theorem for $P_N$; the proof for $P_N^{(k)}$ is done analogously.
For definiteness and convenience, we define $\frac{\log n}{\log p^+(n)}|_{n=1}=0$.
Let
$$
J^+_N=\max\{j\in[N]:X_{p_j}\neq0\},
$$
 with $\max\emptyset$ defined to be 0.
By Proposition \ref{IN},
$\frac{\log n}{\log p^+(n)}$ under $P_N$ is equal in distribution to $\frac{1_{\{J^+_N\neq0\}}}{\log p_{J^+_N}}\sum_{j=1}^N X_{p_j}\log p_j$.
On $\{J^+_N\neq0\}$, we write
$$
\frac1{\log p_{J^+_N}}\sum_{j=1}^N X_{p_j}\log p_j=\frac1{\log p_{J^+_N}}\sum_{j=1}^{J^+_N-1}X_{p_j}\log p_j+X_{p_{J^+_N}}.
$$
As noted in the paragraph containing \eqref{DNeq}, $J^+_N\to\infty$ a.s. as $N\to\infty$.
Also, by the independence of $\{X_{p_j}\}_{j=1}^\infty$, we have $\sum_{j=1}^{J^+_N-1}X_{p_j}\log p_j|\{J^+_N=j_0\}\stackrel{\text{dist}}{=}\sum_{j=1}^{j_0-1}X_{p_j}\log p_j$.
 Thus, it follows from Theorem \ref{Dickman} that
$\frac1{\log J^+_N}\sum_{j=1}^{J^+_N-1}X_{p_j}\log p_j$
 converges weakly to the Dickman distribution. By the PNT, $p_{J^+_N}\sim J^+_N\log J^+_N$; thus also
$\frac1{\log p_{J^+_N}}\sum_{j=1}^{J^+_N-1}X_{p_j}\log p_j$ a.s.
converges weakly to the Dickman distribution.

Note that $X_{p_{J^+_N}}$ conditioned on $\{J^+_N=j_0\}$ is distributed as
$X_{p_{j_0}}$ conditioned on $\{X_{p_{j_0}}\ge1\}$. A trivial calculation shows that $X_{p_{j_0}}$ conditioned on $\{X_{p_{j_0}}\ge1\}$ converges weakly to 1 as $j_0\to\infty$;
thus,  $X_{p_{j_N}}$ converges  weakly to 1. Consequently, $\frac{\log n}{\log p^+(n)}$ under $P_N$ converges weakly to $D+1$ as $N\to\infty$.
\hfill $\square$

\section{Proof of Theorem \ref{BuchstabforPN}}
As noted after the statement of the theorem, we will prove \eqref{Buchstab} with $s\omega(s)$ replaced by $\sum_{L=1}^{[s]}\Lambda_L(s)$, where $\Lambda_L$ is as in \eqref{Lambda}.
That is, we will prove that
\begin{equation}\label{goal}
P_N(\frac{\log n}{\log p^+(n)}\le s)\sim (e^{-\gamma}\log\log N)\frac{\sum_{L=1}^{[s]}\Lambda_L(s)}{\log N},\ s\ge1.
\end{equation}
We will first prove \eqref{goal}  for $s\in[1,2] $, then for $s\in[2,3]$, and then for $s\in[3,4]$. After treating these three particular cases, an inductive argument for the general case of $s\in[L,L+1]$ will be explained
succinctly.

For definiteness and convenience, we define $\frac{\log n}{\log p^-(n)}|_{n=1}=0$.
Of course, $\frac{\log n}{\log p^-(n)}\ge1$, for $n\ge2$.
Let
$$
J^-_N=\min\{j\in[N]:X_{p_j}\neq0\},
$$
 with $\min\emptyset$ defined to be 0.
Note that by \eqref{Mertensform},
\begin{equation}\label{trivialcase}
P(\frac{\log n}{\log p^-(n)}<1)=P(J_N^-=0)=C_N^{-1}\sim \frac{e^{-\gamma}}{\log N}.
\end{equation}
By Proposition \ref{IN},
$\frac{\log n}{\log p^-(n)}$ under $P_N$ is equal in distribution to $\frac{1_{\{J^-_N\neq0\}}}{\log p_{J^-_N}}\sum_{j=1}^N X_{p_j}\log p_j$.
Thus, we have
\begin{equation}\label{keyformula}
\begin{aligned}
&P_N(L\le\frac{\log n}{\log p^-(n)}\le s)=\\
&\sum_{a=1}^NP\Big(L\log p_a\le\sum_{j=a}^N  X_{p_j}\log p_j\le s\log p_a|J_N^-=a\Big)P(J_N^-=a),\ \text{for}\ L\in\mathbb{N},
\end{aligned}
\end{equation}
and
\begin{equation}\label{JNprob}
P(J_N^-=a)=\frac1{p_a}\prod_{j=1}^{a-1}(1-\frac1{p_j}).
\end{equation}

Under
the conditioning $\{J_N^-=a\}$, the  random variables $\{X_{p_j}\}_{j=a}^N$ are still independent,
and for $j>a$, $X_{p_j}$ is distributed as before, namely according to Geom($1-\frac1{p_j}$); however $X_{p_a}$ is now distributed
as a Geom($1-\frac1{p_a}$) random variable  conditioned to be positive.

 Consider first $L=1$ and  $s\in[1,2]$. For $s\neq2$, the inequality $\log p_a\le\sum_{j=a}^N X_{p_j}\log p_j\le s\log p_a$
in \eqref{keyformula} under the conditional probability $P(~\cdot~|J_N^-=a)$  will hold if and only if $X_{p_a}=1$ and $X_{p_j}=0$, for $a+1\le j\le N$.
For $s=2$ it will hold if and only if  $X_{p_a}$ is equal to either 1 or 2 and $X_{p_j}=0$, for $a+1\le j\le N$.
Thus, we have
\begin{equation}\label{s12}
\begin{aligned}
&P\Big(\log p_a\le \sum_{j=a}^N X_{p_j}\log p_j\le s\log p_a|J_N^-=a\Big)=\\
 &\begin{cases}
 \prod_{j=a}^N(1-\frac1{p_j}),\ s\in[1,2);\\
 \prod_{j=a}^N(1-\frac1{p_j})+\frac1{p_a}\prod_{j=a}^N(1-\frac1{p_j}),\ s=2.\end{cases}
\end{aligned}
\end{equation}
From \eqref{trivialcase}-\eqref{s12}, along with \eqref{Mertensform} and \eqref{Mertens2} and the fact that $\Lambda_1(s)\equiv1$ for $s\in[1,2]    $,
we obtain
\begin{equation}\label{finals12}
P_N(\frac{\log n}{\log p^-(n)}\le s)\sim C_N^{-1}\sum_{a=1}^N\frac1{p_a}\sim (e^{-\gamma}\log\log N)~\frac{\Lambda_1(s)}{\log N},\ s\in[1,2].
\end{equation}

Now consider $L=2$ and $s\in[2,3]$.
 Let
$$
J_{a,1}(s)=\max\{j:p_j\le p_a^{s-1}\}.
$$
(Note that $J_{a,1}(s)\ge a$, for $s\ge2$.) Then for $s\in[2,3)$,
 the inequality $2\log p_a\le\sum_{j=a}^N X_{p_j}\log p_j\le s\log p_a$ in \eqref{keyformula} under the conditional probability $P(~\cdot~|J_N^-=a)$
 will hold if and only if either $X_{p_a}=2$ and $X_{p_j}=0$ for $a+1\le j\le N$, or
 $X_{p_a}=1$, $X_{p_j}=1$ for exactly one  $j$ satisfying $a+1\le j\le J_{a,1}(s)\wedge N$, and
 $X_{p_j}=0$ for all other $j$ satisfying $a+1\le j\le N$. Thus,
we have
\begin{equation}\label{s23}
\begin{aligned}
& P\Big(2\log p_a\le \sum_{j=a}^N X_{p_j}\log p_j\le s\log p_a|J_N^-=a\Big)=\\
 &\frac1{p_a}\prod_{j=a}^N(1-\frac1{p_a})+\sum_{l=a+1}^{J_{a,1}(s)\wedge N}\frac1{p_l}\prod_{j=a}^N(1-\frac1{p_j}),
 \ s\in[2,3),
 \end{aligned}
 \end{equation}
where, of course, the sum on the right hand side above is interpreted as 0 if $J_{a,1}(s)=a$.
For the case $s=3$, there is also the possibility of $X_{p_a}=3$ and $X_{p_j}=0$ for $a+1\le j\le N$. The  $P(~\cdot~|J_N^-=a)$-probability of this is
$\frac1{p_a^2}\prod_{j=a}^N(1-\frac1{p_a})$.
Thus, with $s=3$, \eqref{s23} has the additional term $\frac1{p_a^2}\prod_{j=a}^N(1-\frac1{p_a})$ on the right hand side. However, this term  does not contribute to the leading order
asymptotics.
From \eqref{keyformula}, \eqref{JNprob} and  \eqref{s23},
we obtain
\begin{equation}\label{forfinals23}
P_N(2\le\frac{\log n}{\log p^-(n)}\le s)=C_N^{-1}\big(\sum_{a=1}^N\frac1{p^2_a}
+\sum_{a=1}^N\frac1{p_a}\sum_{l=a+1}^{J_{a,1}(s)\wedge N}\frac1{p_l}\big),\ s\in[2,3).
\end{equation}

Since $p_a\sim a\log a$ as $a\to\infty$, it follows that
\begin{equation}\label{Ja1asymp}
J_{a,1}(s)\log J_{a,1}(s)\sim(a\log a)^{s-1},\  \text{as}\ a\to\infty.
\end{equation}
Taking the logarithm of each side in \eqref{Ja1asymp}, we obtain
\begin{equation}\label{logfraclim}
\lim_{a\to\infty}\frac{\log J_{a,1}(s)}{\log a}=s-1.
\end{equation}
Using Mertens' second theorem in the form \eqref{Mertens2nd} along with the fact that $p_j\sim j\log j$, we have
\begin{equation}\label{logfraclog}
\sum_{l=a+1}^{J_{a,1}(s)}\frac1{p_l}\sim \log\log\big(J_{a,1}(s)\log J_{a,1}(s)\big)-\log\log( a\log a)\sim
\log\frac{\log J_{a,1}(s)}{\log a},\ \text{as}\ a\to\infty,
\end{equation}
and thus, by \eqref{logfraclim},
\begin{equation}\label{almostfinal23}
\lim_{a\to\infty}\sum_{l=a+1}^{J_{a,1}(s)}\frac1{p_l}=\log(s-1).
\end{equation}

Now choose any $b\in(0,\frac1s)$. Then $(N^b\log  N^b)^s<N$ for all large  $N$.
By \eqref{Ja1asymp},
\begin{equation}\label{b}
J_{a,1}(s)\le N,\ \text{for}\ a\le N^b\ \text{and sufficiently large}\ N.
\end{equation}
By Mertens' second theorem in the form \eqref{Mertens2}, we have
\begin{equation}\label{breakingupsum}
\sum_{a=1}^N\frac1{p_a}=\sum_{a=1}^{N^b}\frac1{p_a}+O(1)\sim\log\log N.
\end{equation}
From \eqref{almostfinal23}-\eqref{breakingupsum}, we obtain
\begin{equation}\label{paps}
\sum_{a=1}^N\frac1{p_a}\sum_{l=a+1}^{J_{a,1}(s)\wedge N}\frac1{p_l}\sim
\sum_{a=1}^{N^b}\frac1{p_a}\sum_{l=a+1}^{J_{a,1}(s)}
\frac1{p_l}\sim(\log\log N)\log(s-1).
\end{equation}
Recalling the asymptotic behavior of $C_N$, recalling from \eqref{Lambda} that $\Lambda_2(s)=\log(s-1)$ for $s\ge2$, and using
 \eqref{forfinals23} and \eqref{paps}, we conclude that
\begin{equation}\label{finals23}
P_N(2\le\frac{\log n}{\log p^-(n)}\le s)\sim(e^{-\gamma}\log\log N)\frac{\Lambda_2(s)}{\log N},\ s\in[2,3],
\end{equation}
where the inclusion of the right endpoint $s=3$ follows from the remarks made after \eqref{s23}.
From  \eqref{finals12} with $s=2$ and \eqref{finals23}, along with the fact that $\Lambda_1(s)\equiv1$,  we obtain
\begin{equation}\label{finals13}
P_N(\frac{\log n}{\log p^-(n)}\le s)\sim(e^{-\gamma}\log\log N)\frac{\Lambda_1(s)+\Lambda_2(s)}{\log N}, \ s\in[2,3].
\end{equation}

Now consider $L=3$ and $s\in[3,4]$. In fact we will work with $s\in[3,4)$ since the case $s=4$ is slightly different but leads to the same
asymptotics, similar to the remarks after \eqref{s23}.
Then the inequality $3\log p_a\le\sum_{j=a}^N X_{p_j}\log p_j\le s\log p_a$ in \eqref{keyformula} under the conditional probability $P(~\cdot~|J_N^-=a)$
 will hold if and only if one of the following four
 situations obtains:
\begin{equation}\label{situations}
\begin{aligned}
& (1)\ X_{p_a}=3;  X_{p_j}=0, \ \text{for}\ a+1\le j\le N.\\
& (2)\
 X_{p_a}=2; X_{p_j}=1\ \text{ for exactly one}\  j \ \text{satisfying} \ a+1\le j\le J_{a,1}(s-1)\wedge N;\\
&X_{p_j}=0 \ \text{for all other}\ j \ \text{satisfying}\  a+1\le j\le N.\\
&(3)\ X_{p_a}=1; X_{p_j}=1\ \text{ for exactly one}\  j \ \text{satisfying}\ J_{a,1}(3)< j\le J_{a,1}(s)\wedge N;\\
& X_{p_j}=0\ \text{ for all other}\ j \ \text{satisfying}\ a+1\le j\le N.\\
& (4)\ X_{p_a}=1;\ \text{there exist}\ j_1,j_2,\ \text{ satisfying} \ a+1\le j_1\le j_2\le N\ \text{ and}\ p_{j_1}p_{j_2}\le p_a^{s-1},\\
& \text{ such that}\ X_{j_1}=X_{j_2}=1,
\ \text{ if}\ j_1\neq j_2\ \text{ and}\ X_{j_1}=2\ \text{if}\ j_1=j_2;\\
& X_{p_j}=0\ \text{ for all other}\ j\ \text{ satisfying}\ a+1\le j\le N.
\end{aligned}\end{equation}
Because $\sum_{a=1}^\infty \frac1{p_a^2}<\infty$,
the probabilities from situations (1) and (2) in \eqref{situations} do not contribute to the leading order asymptotics
of  $P_N(3\le\frac{\log n}{\log p^-(n)}\le s)$,
 just as in the case $L=2$ and $s\in[2,3)$, the probability from the case $X_{p_a}=2$ did not
contribute to the leading order asymptotics there. (The  contribution there from the case $X_{p_a}=2$ is the term
$C_N^{-1}\sum_{a=1}^N\frac1{p_a^2}$ in \eqref{forfinals23}.)

The analysis of the contribution  from situation (3) in \eqref{situations} follows the same line of analysis as above when $L=2$ and $s\in[2,3)$ for the case
 $X_{p_a}=1$, $X_{p_j}=1$ for exactly one  $j$ satisfying $a+1\le j\le J_{a,1}(s)\wedge N$, and
 $X_{p_j}=0$ for all other $j$ satisfying $a+1\le j\le N$. The difference is that there one had
 $X_{p_j}=1$ for exactly one  $j$ satisfying $a+1\le j\le J_{a,1}(s)\wedge N$, while here one has
 $X_{p_j}=1$ for exactly one  $j$ satisfying $J_{a,1}(3)< j\le J_{a,1}(s)\wedge N$.
Thus, whereas the corresponding contribution there was the term $\sum_{a=1}^N\frac1{p_a}\sum_{l=a+1}^{J_{a,1}(s)\wedge N}\frac1{p_l}$ in \eqref{forfinals23}, the contribution here will be
$\sum_{a=1}^N\frac1{p_a}\sum_{l=J_{a,1}(3)+1}^{J_{a,1}(s)\wedge N}\frac1{p_l}$.
Similar to \eqref{logfraclog}, we have
$\sum_{l=J_{a,1}(3)+1}^{J_{a,1}(s)\wedge N}\frac1{p_l}\sim\log\frac{\log J_{a,1}(s)}{\log J_{a,1}(3)}$,
and from \eqref{logfraclim} we have
$\lim_{a\to\infty}\frac{\log J_{a,1}(s)}{\log J_{a,1}(3)}=\frac{s-1}{3-1}=\frac{s-1}2$.
Thus, similar to \eqref{paps}, we obtain
\begin{equation}\label{papsagain}
\sum_{a=1}^N\frac1{p_a}\sum_{l=J_{a,1}(3)+1}^{J_{a,1}(s)\wedge N}\frac1{p_l}\sim
\sum_{a=1}^{N^b}\frac1{p_a}\sum_{l=J_{a,1}(3)+1}^{J_{a,1}(s)}
\frac1{p_l}\sim(\log\log N)\big(\log (s-1)-\log2\big).
\end{equation}
And finally, similar to \eqref{finals23}, the contribution to $P_N(3\le\frac{\log n}{\log p^-(n)}\le s)$    from situation (3), which we denote by $\rho_3(s)$, satisfies
\begin{equation}\label{rho3}
\rho_3(s)\sim(e^{-\gamma}\log\log N)\frac{\Lambda_2(s)-\Lambda_2(3)}{\log N}, \ s\in[3,4],
\end{equation}
where the inclusion of the right endpoint $s=4$ follows from the remarks made at the beginning of the treatment of the case $s\in[3,4]$.

We know analyze the contribution    from situation (4) in  \eqref{situations}.
From \eqref{keyformula} and \eqref{JNprob}, the contribution  to $P_N(3\le\frac{\log n}{\log p^-(n)}\le s)$    from situation (4), which we will denote by $\rho_4(s)$, is
\begin{equation}\label{4contrib}
\rho_4(s)=C_N^{-1}\sum_{a=1}^N\frac1{p_a}\sum_{\stackrel{a+1\le j_1\le j_2\le N}{ p_{j_1}p_{j_2}\le p_a^{s-1}}}\frac1{p_{j_1}p_{j_2}}.
\end{equation}
Define
$$
J_a(s,j_1)=\max\{j:p_j\le\frac{p_a^{s-1}}{p_{j_1}}\},\ \ \   J_{a,2}(s)=\max\{j:p_j^2\le p_a^{s-1}\}.
$$
Then
\begin{equation}\label{rho4eq}
\rho_4(s)=C_N^{-1}\sum_{a=1}^N\frac1{p_a}\sum_{j_1=a+1}^{J_{a,2}(s)\wedge N}\frac1{p_{j_1}}\sum_{j_2=j_1}^{J_a(s,j_1)\wedge N}\frac1{p_{j_2}}.
\end{equation}

Since $p_j\sim j\log j$, it follows that
$J_{a,2}(s)\log J_{a,2}(s)\sim(a\log a)^{\frac{s-1}2}$, as $a\to\infty$.
Taking the logarithm of both sides above, it follows that
$\log J_{a,2}(s)\sim\frac{s-1}2\log a$ as $a\to\infty$.
Thus
\begin{equation}\label{asympJa2}
J_{a,2}(s)\sim \frac2{s-1}a^{\frac{s-1}2}(\log a)^{\frac{s-3}2},\ \text{as}\ a\to\infty.
\end{equation}

Consider now $J_a(s,j_1)$, for $a+1\le j_1\le J_{a,2}(s)$.
Similarly as in the  above paragraph, it follows that
$J_a(s,j_1)\log J_a(s,j_1)\sim\frac{(a\log a)^{s-1}}{j_1\log j_1}$, as $j_1,a\to\infty$.
Since $j_1\le J_{a,2}(s)$, it follows from \eqref{asympJa2} that
$j_1=o(a^{s-1})$. Thus,
taking the logarithm of both sides above,
we have
\begin{equation}\label{logsim}
\log J_a(s,j_1)\sim(s-1)\log a-\log j_1, \ \text{as}\ j_1,a\to\infty.
\end{equation}
Therefore,
\begin{equation}\label{asympJasj1}
J_a(s,j_1)\sim\frac{a^{s-1}(\log a)^{s-1}}{j_1\log j_1\big((s-1)\log a-\log j_1\big)},\ \text{as}\ j_1,a\to\infty.
\end{equation}

In light of \eqref{asympJa2} and \eqref{asympJasj1}, we can choose $b\in(0,1)$, depending on $s$,
such that $J_a(s,j_1)\le N$ and $J_{a,2}(s)\le N$, for all $a\le N^b$ and all sufficiently large $N$.
Thus, from \eqref{breakingupsum} and \eqref{rho4eq} , we have, similar to the first asymptotic equivalence in \eqref{paps},
\begin{equation}\label{papjpj}
\rho_4(s)\sim C_N^{-1}\sum_{a=1}^{N^b}\frac1{p_a}\sum_{j_1=a+1}^{J_{a,2}(s)}
\frac1{p_{j_1}}\sum_{j_2=j_1}^{J_a(s,j_1)}\frac1{p_{j_2}}.
\end{equation}

By \eqref{Mertens2} and \eqref{logsim}, we have
\begin{equation}\label{innersum}
\begin{aligned}
&\sum_{j_2=j_1}^{J_a(s,j_1)}\frac1{p_{j_2}}\sim\log\frac{\log J_a(s,j_1)}{\log j_1}\sim
\log\frac{(s-1)\log a-\log j_1}{\log j_1}=\log\big((s-1)\frac{\log a}{\log j_1}-1\big), \\
&\text{as}\
j_1, a\to\infty.
\end{aligned}
\end{equation}
Using \eqref{innersum}, \eqref{asympJa2} and  the fact that $p_j\sim j\log j$ as $j\to\infty$,
we  have
\begin{equation}\label{integralform}
\begin{aligned}
&\sum_{j_1=a+1}^{J_{a,2}(s)}
\frac1{p_{j_1}}\sum_{j_2=j_1}^{J_a(s,j_1)}\frac1{p_{j_2}}\sim\sum_{j_1=a+1}^{J_{a,2}(s)}
\frac1{j_1\log j_1}\log\big((s-1)\frac{\log a}{\log j_1}-1\big)\sim\\
&\int_a^{J_{a,2}(s)}\frac1{x\log x}\log\big((s-1)\frac{\log a}{\log x}-1\big)dx\sim\\
&\int_a^{a^{\frac{s-1}2}}\frac1{x\log x}\log\big((s-1)\frac{\log a}{\log x}-1\big)dx, \ \text{as}\ a\to\infty,
\end{aligned}
\end{equation}
where the final asymptotic equivalence follows from the iterated logarithmic growth rate
of the indefinite integral of the integrand appearing in the equation.
Making the substitution
$$
x=a^{\frac{s-1}{u_2}}
$$
reveals that   the second integral in \eqref{integralform}
 is in fact independent of $a$. We obtain
 \begin{equation}\label{integralindepofa}
\begin{aligned}
&\int_a^{a^{\frac{s-1}2}}\frac1{x\log x}\log\big((s-1)\frac{\log a}{\log x}-1\big)dx= \int_2^{s-1}\frac{du_2}{u_2}\log(u_2-1)=\\
&\int_2^{s-1}\int_1^{u_2-1}\frac{du_1}{u_1}\frac{du_2}{u_2}=\Lambda_3(s).
\end{aligned}
\end{equation}
From \eqref{integralform} and \eqref{integralindepofa}, we conclude that
\begin{equation}\label{alimitlambda3}
\lim_{a\to\infty}\sum_{j_1=a+1}^{J_{a,2}(s)}
\frac1{p_{j_1}}\sum_{j_2=j_1}^{J_a(s,j_1)}\frac1{p_{j_2}}=\Lambda_3(s).
\end{equation}
Thus, recalling the asymptotic behavior of $C_N$,  from \eqref{alimitlambda3}, \eqref{papjpj} and \eqref{breakingupsum} we conclude  that
\begin{equation}\label{rho4}
\rho_4(s)\sim   (e^{-\gamma}\log\log N)\frac{\Lambda_3(s)}{\log N},\ s\in[3,4],
\end{equation}
where the inclusion of the right endpoint $s=4$ follows from the remarks made at the beginning of the treatment of the case $s\in[3,4]$.
From \eqref{rho3} and \eqref{rho4}, we conclude that
\begin{equation}\label{finals34}
P_N(3\le\frac{\log n}{\log p^-(n)}\le s)\sim(e^{-\gamma}\log\log N)
\frac{\big(\Lambda_2(s)-\Lambda_2(3)\big)+\Lambda_3(s)}{\log N},\ s\in[3,4].
\end{equation}
From \eqref{finals13} with $s=3$ and \eqref{finals34}, and recalling that $\Lambda_1(s)\equiv1$, we have
\begin{equation*}
P_N(\frac{\log n}{\log p^-(n)}\le s)\sim(e^{-\gamma}\log\log N)\frac{\Lambda_1(s)+\Lambda_2(s)+\Lambda_3(s)}{\log N},
\ s\in[3,4].
\end{equation*}
\medskip

We now consider the general case that $s\in [L,L+1]$.
By induction, we have
\begin{equation}\label{induction}
P_N(\frac{\log n}{\log p^-(n)}\le s)\sim(e^{-\gamma}\log\log N)\frac{\sum_{l=1}^{[s]}\Lambda_l(s)}{\log N},\ s\le L.
\end{equation}
 Making a list similar to \eqref{situations}, and analyzing the situations as was done there, one concludes that
the situations with $X_{p_a}\ge2$ do not contribute to the leading order asymptotics of $P_N(L\le\frac{\log n}{\log p^-(n)}\le s)$, while  the situations
with $X_{p_a}=1$ do contribute. When $X_{p_a}=1$,  we obtain $L-1$ situations, with
all but one of them of the form already treated in the case of  $L-1$.  (In \eqref{situations}, where $L=3$, there were  2 such situations---labeled there (3) and (4), and one
of them, namely (3), was of the form already treated for $L=2$.) Thus, by induction and by the argument used to show that the contribution from situation
(3) in \eqref{situations} is as it appears in \eqref{rho3},
  these terms will give asymptotic contributions
\begin{equation}\label{othersit}
(e^{-\gamma}\log\log N)\frac{\Lambda_1(s)-\Lambda_1(L)}{\log N},\ldots, (e^{-\gamma}\log\log N)\frac{\Lambda_{L-1}(s)-\Lambda_{L-1}(L)}{\log N}.
\end{equation}

We now look at the new situation that arises; namely
the one in which  $X_{p_a}=1$ and there exist $j_1,\ldots, j_{L-1}$ satisfying $a+1\le j_1\le\cdots\le j_{L-1}\le N$ and $\prod_{i=1}^{L-1}p_{j_i}\le p_a^{s-1}$,
such that for $j\in\{a+1,\ldots, N\}$, $X_j$ is equal to the number of times $j$ appears among the $\{j_i\}_{i=1}^{L-1}$.
From \eqref{keyformula} and \eqref{JNprob}, the contribution  to $P_N(L\le\frac{\log n}{\log p^-(n)}\le s)$    from this situation, similar to \eqref{4contrib} in the case $L=3$, is
\begin{equation}\label{Lcontrib}
C_N^{-1}\sum_{a=1}^N\frac1{p_a}\sum_{\stackrel{a+1\le j_1\le\cdots\le j_{L-1}\le N}{\prod_{i=1}^{L-1}p_{j_i}\le p_a^{s-1}}}\frac1{\prod_{i=1}^{L-1}p_{j_i}}.
\end{equation}
An analysis analogous to that implemented between
\eqref{4contrib} and
\eqref{alimitlambda3} gives
\begin{equation}\label{alimit}
\lim_{a\to\infty}\sum_{\stackrel{a+1\le j_1\le\cdots\le j_{L-1}\le N}{\prod_{i=1}^{L-1}p_{j_i}\le p_a^{s-1}}}\frac1{\prod_{i=1}^{L-1}p_{j_i}}
=\int_{L-1}^{s-1}\int_{L-2}^{u_{L-1}-1}\cdots\int_1^{u_2-1}\prod_{j=1}^{L-1}
\frac{du_j}{u_j}=\Lambda_L(s).
\end{equation}
From \eqref{Lcontrib} and \eqref{alimit} it follows that the contribution to the leading order asymptotics
of $P_N(L\le\frac{\log n}{\log p^-(n)}\le s)$
from  this situation is
$(e^{-\gamma}\log\log N)\frac{\Lambda_L(s)}{\log N}$.
We conclude from this and \eqref{othersit} that
\begin{equation}\label{newstep}
P_N(L\le\frac{\log n}{\log p^-(n)}\le s)\sim (e^{-\gamma}\log\log N)\frac{\Lambda_L(s)+\sum_{l=1}^{L-1}\big(\Lambda_l(s)-\Lambda_l(L)\big)}{\log N},\ s\in[L,L+1].
\end{equation}
From \eqref{induction} with $s=L$ and from \eqref{newstep}, we conclude that
$$
P_N(L\frac{\log n}{\log p^-(n)}\le s)\sim(e^{-\gamma}\log\log N)\frac{\sum_{l=1}^{L}\Lambda_l(s)}{\log N},\ s\in[L,L+1].
$$
This completes the proof of \eqref{goal}.
\hfill $\square$

\section{Appendix: Proof of Proposition \ref{Skbetap}}
For notational convenience, we will work with $p$ instead of $p_j$.
The proof is via the inclusion-exclusion principle along with the fact that
$D_{\text{nat}}(S_k)=\frac1{\zeta(k)}$, where
$S_k$ denotes the $k$-free integers,
as was noted with a reference in the proof of Proposition
\ref{nat-log-indep}.
Recall that $1\le l<k$.
We have
$$
\begin{aligned}
&I_N\equiv|\{n: p\thinspace^l|n,  n\le N, n\in S_k\}|=|\{n_1:n_1\le [\frac N{p\thinspace^l}], n_1\in S_k,p\thinspace^{k-l}\nmid n_1\}|=\\
&|\{n_1:n_1\le [\frac N{p\thinspace^l}], n_1\in S_k\}|-|\{n_1:n_1\le [\frac N{p\thinspace^l}], n_1\in S_k,p\thinspace^{k-l}|n_1\}|\equiv I_{N,1}-I_{N,2}.
\end{aligned}
$$
Similarly,
$$
\begin{aligned}
&I_{N,2}=|\{n_2:n_2\le [\frac N{p\thinspace^k}],n_2\in S_k, p\thinspace^l\nmid n_2\}|=\\
&|\{n_2:n_2\le [\frac N{p\thinspace^k}],n_2\in S_k\}|-|\{n_2:n_2\le [\frac N{p\thinspace^k}],n_2\in S_k, p\thinspace^l|n_2\}|\equiv I_{N,3}-I_{N,4},
\end{aligned}
$$
and
$$
\begin{aligned}
&I_{N,4}=|\{n_3:n_3\le [\frac N{p\thinspace^{k+l}}],n_3\in S_k, p\thinspace^{k-l}\nmid n_3\}|=\\
&|\{n_3:n_3\le [\frac N{p\thinspace^{k+l}}],n_3\in S_k\}|-|\{n_3:n_3\le [\frac N{p\thinspace^{k+l}}],n_3\in S_k,p\thinspace^{k-l}|n_3\}|\equiv I_{N,5}-I_{N,6}.
\end{aligned}
$$
So up to this point, we have
$$
I_N=
I_{N,1}-I_{N,3}+I_{N,5}-I_{N,6}.
$$
Now     $\lim_{N\to\infty}\frac{I_{N,1}}N=\frac1{p\thinspace^l}D_{\text{nat}}(S_k)=\frac1{p\thinspace^l\zeta(k)}$,
$\lim_{N\to\infty}\frac{I_{N,3}}N=\frac1{p\thinspace^k}D_{\text{nat}}(S_k)=\frac1{p\thinspace^k\zeta(k)}$ and
$\lim_{N\to\infty}\frac{I_{N,5}}N=\frac1{p\thinspace^{k+l}}D_{\text{nat}}(S_k)=\frac1{p\thinspace^{k+l}\zeta(k)}$.
Continuing this process of inclusion-exclusion,
we have
$$
I_N=\sum_{m=0}^\infty I_{N,4m+1}-\sum_{m=0}^\infty I_{N,4m+3},
$$
where for each $N$ only a finite number of the summands above are non-zero.
Now
$$
\lim_{N\to\infty}\frac{I_{N,4m+1}}N=\frac1{p\thinspace^{mk+l}\zeta(k)},\ m=0,1,\ldots,
$$
and
$$
\lim_{N\to\infty}\frac{I_{N,4m+3}}N=\frac1{p\thinspace^{(m+1)k}\zeta(k)},\ m=0,1,\ldots.
$$
From this we conclude that
$$
D_{\text{nat}}(\beta_p\ge l, S_k)=\frac1{\zeta(k)}\sum_{m=0}^\infty\frac1{p^{mk+l}}-
\frac1{\zeta(k)}\sum_{m=0}^\infty\frac1{p^{(m+1)k}}=\frac1{\zeta(k)}\frac{\frac1{p^l}-\frac1{p^k}}
{1-\frac1{p^k}}.
$$
Thus,
$D_{\text{nat}}(\beta_p\ge l|S_k)=\frac{\frac1{p^l}-\frac1{p^k}}
{1-\frac1{p^k}}$.
\hfill $\square$

\end{document}